\documentclass[12pt,reqno]{amsart}
\usepackage{setspace}
\usepackage{hyperref}
\usepackage{graphicx}
\usepackage{xcolor}
\usepackage{amsmath}
\usepackage{amsthm,amsfonts}
\usepackage{enumitem}
\usepackage{todonotes}
\usepackage{tikz}

\usepackage[left=2.8cm,top=2.8cm,right=2.8cm,bottom=2.8cm]{geometry}

\newtheorem{theorem}{Theorem}
\newtheorem{corollary}[theorem]{Corollary}
\newtheorem{lemma}[theorem]{Lemma}

\hypersetup{
    colorlinks = false,
    allbordercolors = {white},
}

\tikzstyle{vert}=[shape=circle,draw=black,fill=white, inner sep=.5mm]

\begin{document}
\title{On Meyniel extremal families of graphs}
\author[A.\ Bonato]{Anthony Bonato}
\author[R.\ Cushman]{Ryan Cushman}
\author[T.G.\ Marbach]{Trent G.\ Marbach}
\address[A1, A2, A3]{Department of Mathematics, Ryerson University, Toronto, Canada,  M5B 2K3.}

\email[A1]{(A1) abonato@ryerson.ca}
\email[A2]{(A2) ryan.cushman@ryerson.ca}
\email[A3]{(A3) trent.marbach@ryerson.ca}

\begin{abstract}
We provide new constructions of Meyniel extremal graphs, which are families of graphs with the conjectured largest asymptotic cop number.  Using spanning subgraphs, we prove that there are an exponential number of new Meyniel extremal families with specified degrees. Using a linear programming problem on hypergraphs, we explore the degrees in families that are not Meyniel extremal. We give the best-known upper bound on the cop number of vertex-transitive graphs with a prescribed degree. We find new Meyniel extremal families of regular graphs with large chromatic number, large diameter, and explore the connection between Meyniel extremal graphs and bipartite graphs. 
\end{abstract}


\maketitle

\section{Introduction} 

One of the most challenging directions in the study of the game of Cops and Robbers played on graphs is understanding how large the cop number can be as a function of the number of vertices of the graph. The definition of the game is given at the end of the introduction, and we denote the cop number of a graph $G$ by $c(G).$ \emph{Meyniel's conjecture} states that there is a constant $D>0$ such that for all connected graphs of order $n$, $c(G) \le D \sqrt{n}.$ Despite the known sublinear bounds on the cop number \cite{fkl,lu,ss}, the conjecture remains wide open; it is unknown if the \emph{soft Meyniel's conjecture} is true: there is a constant $D>0$ such that the cop number is bounded above by $Dn^{1-\epsilon},$ where $\epsilon >0.$ For more background on Meyniel's conjecture, see \cite{bairdbonato} and the book \cite{bonato}.

Families of graphs realizing the conjectured asymptotic upper bound are of interest in their own right. Let $I$ be an infinite set of positive integers, and let $\lbrace G_n \rbrace_{n\in I}$ be a family of graphs, where $G_n$ has order $n$. Note that $I$ may be a proper subset of the positive integers, such as the set of prime power integers. We say that $\lbrace G_n \rbrace_{n\in I}$ is a \emph{Meyniel extremal} family if there exists a positive constant $d$ such that for $G_n$ we have $c(G_n) \ge d \sqrt n$ for all $n \in I$. We sometimes abuse notation and refer to Meyniel extremal graphs.

The incidence graphs of projective planes form the earliest known example of a Meyniel extremal family, where in this case $I$ is the set of prime powers \cite{frankl}. Several other families of incidence graphs of combinatorial designs are known to be Meyniel extremal \cite{bb}, as are the incidence graphs of partial affine planes \cite{bairdbonato}. Other Meyniel extremal families include polarity graphs \cite{bb}, $t$-orbit graphs \cite{bb}, and certain families of Cayley graphs \cite{brad,hs}.

In the present work, we provide results that diversify the kind and quantity of known Meyniel extremal families. We provide new lower bounds on the cop number of a graph in Section~2, generalizing results in \cite{af,bb}. We consider spanning subgraphs in Section~3, and apply those results to give an exponential number of new Meyniel extremal families with specified degrees. In Section~4, we utilize a linear programming problem on hypergraphs to explore vertex degrees within families assuming Meyniel's conjecture is false. In Corollary~\ref{bbb}, the best-known upper bound on the cop number of vertex-transitive graphs with a prescribed degree is given. New constructions are given of Meyniel extremal families in Section~5; we find Meyniel extremal families of regular graphs with large chromatic number and large diameter. We consider bipartite double covers to explore the connection between Meyniel extremal families and bipartite graphs. The final section contains conjectures and open problems relating graphs in Meyniel extremal families to their degrees and cycle subgraphs.

To finish this introduction, we  give a brief overview of Cops and Robbers for those readers unfamiliar with the game. \emph{Cops and Robbers} is a game played on a graph $G$. There are two players consisting of a set of \emph{cops} and a single \emph{robber}. The game is played over a sequence of discrete time-steps or \emph{rounds} indexed by nonnegative integers, with the cops going first in round $0$. The cops and robber occupy vertices;
for simplicity, we often identify the player with the vertex they occupy. We refer to the set of cops as $C$ and the robber as $R.$ When a player is ready to move in a round, they must move to a neighboring vertex. Players can \emph{pass}, or remain on their own vertex. Any subset of $C$ may move in a given round.

The cops win the game if, after a finite number of rounds, one of them can occupy the same vertex as the robber.
This situation is called a \emph{capture}. The robber wins if they can evade capture indefinitely. The minimum number of cops required to win is a well-defined positive integer, called the \emph{cop number} of $G,$ written $c(G)$. For additional background on the cop number of a graph, see the book \cite{bonato}.

All graphs we consider are simple, finite, and undirected. We only consider connected graphs, unless otherwise stated. The diameter of a graph $G$ is denoted by $\mathrm{diam}(G)$. For a graph $H,$ a graph is $H$-\emph{free} if it does not contain $H$ as a subgraph. In the case of $H=K_3,$ we say the graph is \emph{triangle-free}. For a positive integer $k$ and a vertex $x$ in a graph $G$, let $N_k(x)$ be the set of vertices of distance $k$ to $x.$ We denote $N_1(x) = N(x),$ and refer to vertices in $N(x)$ as \emph{neighbors} of $x$. For a set $S$ of vertices, $N(S)$ is the set of neighbors of vertices in $S.$ For a vertex $x$ in a graph $G,$ we denote the degree of $x$ by $\mathrm{deg}_G(x);$ we drop the subscript $G$ if it clear from context. For a graph $G,$ let $\delta(G)$ and $\Delta(G)$ be the minimum and maximum degrees in $G,$ respectively. For background on graph theory, see \cite{west}. All logarithms are in base 2, unless otherwise stated.

\section{New lower bounds on the cop number}
We give generalizations of the lower bounds provided by Aigner and Fromme~\cite{af} and Bonato and Burgess~\cite{bb}. The following lemma will be useful in the next section.

\begin{lemma}\label{lem:newlb-K2t}
Let $n$ be a positive integer and $k$ a nonnegative integer with $n \ge k.$  If $G$ is $K_{2,t}$-free for $t\ge 1$ an integer and has $n-k$ vertices of degree at least $D$ and $k$ vertices of degree less than $D$, with $D > k$, then 
$$
c(G) \ge \frac{D-k}{t}.
$$
\end{lemma}

\begin{proof}
Define $G_D$ to be the subgraph induced by the $n-k$ vertices of degree at least $D$, and let $G_k$ be the subgraph induced by the remaining vertices. Suppose that $C$ is the set of cops, and suppose that $|C|<\frac{D-k}{t}$. 

Consider a robber strategy where the robber only moves on vertices in $G_D$, and on any given robber move, will move to any neighboring vertex that is not adjacent to a cop. Suppose the robber is on vertex $u$ on the robber's turn, which we assume does not contain a cop. There are at most $k$ neighbors of $u$ in $G_k$ that the robber refuses to move to, leaving at least $D-k$ neighbors that the robber can move to. Each cop can be adjacent to at most $t-1$ vertices among these $D-k$ vertices and can be on another one vertex among these, and so there are at most $|C|t$ vertices among the $D-k$ vertices that the robber cannot move to. As $|C|t < D-k$, then there is always some vertex in the neighborhood of $u$ that is both in $G_D$ and is not adjacent to any cop. 

This shows that the robber can escape capture, given that the robber is not captured in the first round. 
To see the later, after the cops have chosen their starting positions, let $u'$ be any vertex of $G_D$ that does not contain a cop. By the above analysis, $u'$ has a neighbor in $G_D$ that is not adjacent to a cop, and the robber starts on such a vertex.  
\end{proof}

We have the following consequence of Lemma~\ref{lem:newlb-K2t}, first proven in \cite{bb}.

\begin{corollary}\label{cor22}
Let $t \ge 1$ be an integer. If $G$ is $K_{2,t}$-free, then $c(G) \ge \delta(G)/t$.
\end{corollary}

A Meyniel extremal family $\lbrace G_n \rbrace_{n\in I}$ only requires some constant $d$ such that $c(G_n) \ge d \sqrt{n}$ for all $n\in I$, so we will not require the extra precision provided by the following extension of \cite{af}. However, we include it as a potentially helpful result in other contexts.  

\begin{lemma}\label{lem:newlb-girth}
Let $n$ be a positive integer and $k$ a nonnegative integer with $n \ge k.$ If $G$ has girth at least $5$ and has $n-k$ vertices of degree at least $D$ and $k$ vertices of degree less than $D$, with $D > k$, then 
$$
c(G) \ge D-k.
$$
\end{lemma}
\begin{proof}
Define $G_D$ to be the subgraph induced by the set of $n-k$ vertices of degree at least $D$, and let $G_k$ be the subgraph induced by the remaining vertices. Suppose that $C$ is the set of cops and that $|C| = D-k-1$. 
We first show that the robber can choose a vertex $u$ that is not adjacent to any cop and is in $G_D$. Suppose to the contrary that $C$ is a dominating set for $G_D$. Define $X$ to be the set of $x$ neighbors of $u$ adjacent to $C \cap V(G_D)$ and $X_k$ to be the set of $x_k$ neighbors of $u$ adjacent to $C \cap V(G_k)$.
Similarly, let $Y$ be the $y$ neighbors of $u$ in $V(G_D)\setminus C$ and let $Y_k$ be the $y_k$ neighbors of $u$ in $V(G_k)\setminus C$; see Figure~\ref{fig:lb}.

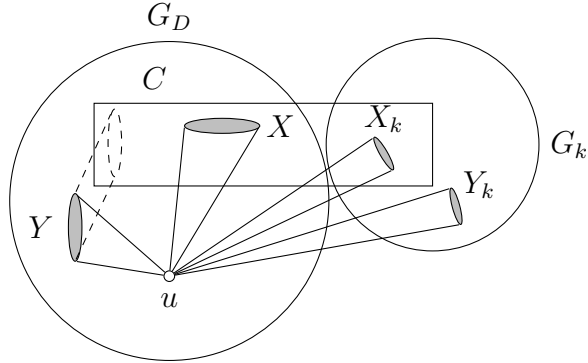
\begin{figure}[h]\label{fig:lb}
\begin{center}
\begin{tikzpicture}  

\node[inner sep=0mm] (r1) at (-1,.2) {};
\node[inner sep=0mm] (r2) at (3.5, 1.3) {};
\draw[black, fill=white, label={$C$}] (r1) rectangle (r2) {};
\node[label={\color{black}$C$}] at (-.2,1.2) {};

\node [shape=circle,draw=black, inner sep=10mm, label=right:{$G_k$}] at (3.5, .75) {};

\node [shape=circle,draw=black, inner sep=15mm, label={$G_D$}] at (0, 0) {};

\node[vert, label=below:$u$] (u) at (0,-1) {};

\begin{scope}[xshift=20]
\node[label=right:{\color{black}$X$}] at (.3,1) {};
\node[inner sep=0mm] (x1) at (.5,1) {};
\node[inner sep=0mm] (x2) at (-.5,1) {};
\draw[black, fill=black!25] (x1) arc (0:360:.5 and .1);
\draw (u) -- (x1);
\draw (u) -- (x2);
\end{scope}

\begin{scope}[xshift=5,yshift=22, scale=.9,rotate=90, dashed]
\node[inner sep=0mm] (ys1) at (.5,1) {};
\node[inner sep=0mm] (ys3) at (.45,1.05) {};
\node[inner sep=0mm] (ys2) at (-.5,1) {};
\draw[black,] (ys1) arc (0:360:.5 and .1);
\end{scope}

\begin{scope}[rotate=90, yshift=10, xshift=-10, scale=.9]
\node[label=left:{\color{black}$Y$}] at (0,1) {};
\node[inner sep=0mm] (y1) at (.5,1) {};
\node[inner sep=0mm] (y2) at (-.5,1) {};
\draw (u) -- (y1);
\draw (u) -- (y2);
\draw[black, fill=black!25] (y1) arc (0:360:.5 and .1);

\end{scope}

\begin{scope}[rotate=-60, yshift= 65,xshift=25, scale=.5]
\node[label={\color{black}$X_k$}] at (0,1) {};
\node[inner sep=0mm] (xk1) at (.5,1) {};
\node[inner sep=0mm] (xk2) at (-.5,1) {};
\draw[black, fill=black!25] (xk1) arc (0:360:.5 and .1);
\draw (u) -- (xk1);
\draw (u) -- (xk2);
\end{scope}

\begin{scope}[rotate=-75, yshift= 90,xshift =30, scale=.5]
\node[label=above:{\color{black}$Y_k$}] at (.5,1.5) {};
\node[inner sep=0mm] (yk1) at (.5,1) {};
\node[inner sep=0mm] (yk2) at (-.5,1) {};
\draw[black, fill=black!25] (yk1) arc (0:360:.5 and .1);
\draw (u) -- (yk1);
\draw (u) -- (yk2);
\end{scope}

\draw[black, dashed] (y1) -- (ys3);
\draw [black, dashed] (y2) -- (ys2);

\end{tikzpicture}
\end{center}
\caption{Sets considered in the proof of Lemma~\ref{lem:newlb-girth}.}
\end{figure}

Since the sets $X,$ $Y,$ and $V(G_k)$ are all disjoint, we obtain that
$$x + y + k \ge   x + y + y_k + x_k = d(u)\ge D.$$
Since $C$ is a dominating set for $G_D$, then there exists vertices of $C$ that dominate $Y$. However, $G$ is $C_4$-free and triangle-free, so there is a unique vertex in $C$ that dominates each vertex of $Y$. Also, since $G$ is triangle-free, none of the vertices in $X$ can be adjacent to vertices in $Y$. Hence, we have that
$$
D-k - 1 = |C| \ge x + y \ge D-k,
$$
a contradiction. Hence, the robber can pick a vertex in round 0.

Now suppose that some number of rounds have been played and the robber occupies a vertex $v$ in $G_D$ that is not adjacent to any cop $C$. Since $G$ is $C_4$-free, each vertex of $C$ can be adjacent to at most one neighbor of $v$. In addition, at most $k$ neighbors $v$ are in $G_k$. Hence, there are at least $D - (D-k-1) - k = 1$ vertices in $G_D$ that are not adjacent to any vertex in $C$. By induction, the robber can stay in $G_D$ indefinitely and evade capture. 
\end{proof}

Taking $k = 0$ yields the classic result of \cite{af}.

\begin{corollary}\label{cor11}
If $G$ has girth at least $5$, then $c(G) \ge \delta(G)$.
\end{corollary}

\section{Meyniel extremal families and spanning subgraphs}\label{sec:spanning}
Many of the known constructions that result in Meyniel extremal families contain vertices with degrees in $\Theta(\sqrt{n})$ and forbid $K_{2,t}$ for integer constant $t \ge 1$. We therefore arrive at the following definition.

A graph family $\lbrace G_n \rbrace_{n\in I}$ is \emph{elementary} if the following properties hold for all $n\in I$.
\begin{enumerate}
\item The graph $G_n$ is $K_{2,t}$-free for $t \ge 1$ an integer constant. 
\item All vertices of $G_n$ have degree in $\Theta(\sqrt{n}).$
\end{enumerate}

In this section, we consider new Meyniel extremal families that are generated by considering subgraphs of members of elementary Meyniel extremal families. This will result in many families that are not elementary Meyniel extremal and many families whose members are nonisomorphic. To make this precise, we define two families $\lbrace G_n \rbrace_{n\in I}$ and $\lbrace J_n\rbrace_{n\in I}$ with the same index set $I$ to be \textit{nonisomorphic} if for each $n \in I$, there is no isomorphism between $G_n$ and $J_n$. Further, we say that the family $\lbrace J_n\rbrace_{n\in I}$ is a \textit{spanning family} of  $\lbrace G_n \rbrace_{n\in I}$ if $J_n$ is a spanning subgraph of $G_n$ for all $n\in I$. 

For a positive integer $r$, an $r$-\emph{factor} of a graph $G$ is a spanning $r$-regular subgraph of $G$, and an $r$-\emph{factorization} of $G$ partitions the edges of $G$ into disjoint $r$-factors.

\begin{lemma}\label{lem:fact}
Fix $0 < \varepsilon < 1$ and let $k$ be a positive-integer. Let $\lbrace G_n \rbrace_{n\in I}$ be a $k$-regular elementary Meyniel extremal family with an $r$-factorization, where $1\le r < k$. We then have that for each $1 \le i \le \lfloor \varepsilon k /r \rfloor$ there exists a family of Meyniel extremal graphs $\lbrace G_{i,n}\rbrace_{n\in I}$ such that $G_{i,n}$ is $(k-ri)$-regular. 
\end{lemma}
\begin{proof}
Fix an $r$-factorization $\lbrace F_1, F_2, \ldots, F_k \rbrace$ of $G_n$.  For all $1 \le i \le \lfloor \varepsilon k/r \rfloor$, define $G_{i, n}$ to have the vertex set $V(G)$ and the edge set 
$$
E(G_{i,n}) = E(G_n) \setminus \left( \bigcup_{j=1}^{i} E(F_i) \right).
$$
Each vertex of $G_n$ loses $ri$ adjacent edges, so $G_{i,n}$ is $(k - ri)$-regular and is $K_{2,t}$-free. Further, we have that $$k - ri \ge k(1-\epsilon) \ge d_1 \sqrt{n},$$
for some constant $d_1$. 
Hence, the family $\lbrace G_{i,n} \rbrace_{n\ge 1}$ is elementary Meyniel extremal by Corollary~\ref{cor22} and its members are $(k - i)$-regular. 
\end{proof}

Lemma~\ref{lem:fact} can be applied to families of regular bipartite graphs and families of regular graphs with even degree, which contain 1- and 2-factorizations, respectively. For example, the family consisting of the incidence graphs of the projective plane for each prime power can be used as a starting elementary Meyniel extremal family. 

Lemma~\ref{lem:newlb-K2t} allows us to find families of Meyniel extremal graphs using spanning subgraphs.

\begin{lemma}\label{lem:subgraph}
Let $\lbrace G_n \rbrace_{n\in I}$ be an elementary Meyniel extremal family, let $k$ be a nonnegative integer, and let $J_n$ be a connected spanning subgraph of $G_n$ that contains at least $n-k$ vertices of degree at least $D$ where $D-k = \Theta(\sqrt n)$. We then have that $\lbrace J_n \rbrace_{n\in I}$ is a Meyniel extremal family. 
\end{lemma}
\begin{proof}
Since $\lbrace G_n \rbrace_{n\ge 1}$ is an elementary Meyniel extremal family, we may apply Lemma~\ref{lem:newlb-K2t} to obtain that $$c(J_n) \ge (D - k)/{t} \ge d\sqrt n/t$$ for some constant $d$ (recall by the definition of an elementary Meyniel extremal family, $t$ is constant). 
\end{proof}

We have the following corollary.

\begin{corollary}\label{cor:vec}
Fix $0 < \varepsilon < 1$. Let $\lbrace G_n \rbrace_{n\in I}$ be a family of $C_4$-free graphs and let the degrees of $G_n$ be $\Theta(\sqrt n)$ and $\delta_n = \delta(G_n) $. For each $\mathbf x = (x_i)_{i=1}^{\lceil\varepsilon \delta_n\rceil}$, there exists a Meyniel extremal family $\lbrace G_{\mathbf{x}, n}\rbrace_{n\in I}$ with the following properties: 
\begin{enumerate}
\item $G_{\mathbf{x}, n}$ is a spanning subgraph of $G_n$; 
\item $G_{\mathbf{x}, n}$ contains vertices $v_i \in V(G_n)$ with $0 \le x_i \le \mathrm{deg}_{G_n}(v_i)-3$; and
\item  $G_{\mathbf{x},n}$ has
$
\mathrm{deg}_{G_{\mathbf{x},n}}(v_i) = \mathrm{deg}_{G_n}(v_i) - x_i
$
for $1 \le i \le \lceil \varepsilon \delta_n\rceil$.  

\end{enumerate}
\end{corollary}
\begin{proof}
Fix $\mathbf x = (x_i)_{i=1}^{\lceil \varepsilon \delta_n\rceil}$. Choose a vertex $v\in V(G_n)$ of minimum degree. Note that since $G_n$ is $C_4$-free, every pair of vertices in $N_1(v)$ have no common neighbors in $N_2(v)$. To see this, if two vertices $x,y\in N_1(v)$ have a common neighbor $z$ in $N_2(v)$, then we obtain a subgraph $X$ isomorphic to $C_4$ with edges $vx$, $xz$, $zy$, and $yv$.  
In addition, each vertex $x\in N_1(v)$ has at most one neighbor in $N_1(v)$. Otherwise, if $y,z \in N_1(v)$ are both neighbors of $x$, then $X$ violates the $C_4$-free property. Hence, each of the $\delta_n$ vertices in $N_1(v)$ has at least $\mathrm{deg}_{G_n}(v_i)-2$ unique neighbors in $N_2(v)$.

Now choose $\lceil \varepsilon \delta_n\rceil$ vertices from $S \subseteq N_1(v)$ and label them $v_1, v_2, \ldots, v_{\lceil \varepsilon \delta_n \rceil}$ and obtain $G_{\mathbf{x},n}$ from $G_n$ by deleting $r_i$ edges incident with $v_i$ and in $N_2(v)$. The graph $G_{\mathbf{x},n}$ is connected since the only edges removed are between $S$ and $N_2(v)$ and each vertex in $S$ has degree at least $1$ into $N_2(v)$. Thus, $\lbrace G_{\mathbf{x},n} \rbrace_{n\in I}$ is Meyniel extremal by Lemma~\ref{lem:subgraph} and contains vertices of the desired degree.
\end{proof}

Notice that Corollary~\ref{cor:vec} applies to any elementary Meyniel extremal graphs that have graphs with girth at least $5$. One application of Lemma~\ref{lem:subgraph} is to construct large families of nonisomorphic, Meyniel extremal, spanning families using a similar approach to the proof of Corollary~\ref{cor:vec}. We will require the following lemma. 
\begin{lemma}\label{lem:H}
Let $J = (A,B)$ be a bipartite graph where each of the $|A|=a$ vertices of $A$ has degree at least $d$ and shares no common neighbors, with $d \le a$. Let $X$ count the number of distinct spanning subgraphs of $J.$ We then have that 
$$
X \ge \binom{a+d-1}{d-1} .
$$
If $a(n) = \Theta(d(n))$, then for $H(x)=-x \log_2 x - (1-x) \log_2 (1-x)$ the binary entropy function, we have that
$$
X \sim 2^{ (a+d-1)H((d-1)/(a+d-1))}.
$$

\end{lemma}
\begin{proof}
Label the vertices of $A$ as $\lbrace v_1, v_2, \ldots, v_a \rbrace$. For each $v_i \in A$, delete edges incident to $v_i$ until it has exactly $d$ neighbors. Define for each vector $\mathbf x = (x_i)_{i=1}^a\in \lbrace 0, 1, 2,\ldots  d-1\rbrace^a$ the graph $J_{\mathbf x}$ obtained by deleting $x_i$ edges incident with $v_i$.

Define  $g_i(\mathbf x)$ be the number of coordinates of $\mathbf x$ that contain $i$, for $0 \le i \le d-1$. For two vectors $\mathbf{x, y}$ if there exists $i$ such that $g_i(\mathbf x) \not= g_i(\mathbf y)$, then $J_{\mathbf x}$ and $J_{\mathbf y}$ are nonisomorphic since they will have different degree sequences. Hence, the number of nonisomorphic spanning subgraphs of $J$ is at least 
$$
\binom{a+d-1}{d-1} 
$$
since this is the number of different ways to distribute $a$ coordinates among the $d$ variables $g_0, g_1, \ldots, g_{d-1}$.

Since $a = \Theta(d)$, we may estimate this binomial coefficient using Stirling's approximation to achieve: 

$$
\binom{a+d-1}{d-1} \sim \sqrt{\frac{a+d-1}{2\pi (d-1)a}} \left(\frac{a+d-1}{d-1}\right)^{d-1} \left(\frac{a+d-1}{a}\right)^{a}. 
$$
Taking the logarithm yields
$$
\log_2\binom{a+d-1}{d-1} \sim H\left(\frac{d-1}{a+d-1}\right)(a+d-1),
$$
where $H(x) = -x \log_2 x - (1-x) \log_2 (1-x)$ for $x\in (0, 1)$. Notice that since 
$$\beta(n) = \frac{d-1}{a+d-1} = 1 - \frac{1}{1 + \frac{d}{a}-\frac{1}{a}}$$
 has $0<\beta\le{1}/{2}$, then we obtain
$$
\binom{a+d-1}{d-1} \sim 2^{(a+d-1)H((d-1)/(a+d-1))}.
$$
The proof follows.
\end{proof}

\begin{theorem}\label{thm:spanning5}
Fix $0 <\varepsilon < 1$.
Let $G_n$ be a graph on $n$ vertices with girth at least $5$ that has degrees in $\Theta(\sqrt{n})$ and let $\delta_n =\delta(G_n)$. We then have that there exists $r$ pairwise nonisorphic, Meyniel extremal, spanning families of $\lbrace G_n \rbrace_{n \in I}$, where 
$$r \ge \binom{(1+\varepsilon)\delta_n-2}{\delta_n-2} \sim 2^{((1+\varepsilon)\delta_n-2)H(1/(1+\varepsilon))}$$ 
and $H(x)$ is the binary entropy function.
\end{theorem}
\begin{proof}
Let $v$ be a vertex of minimum degree in $G_n$ and choose $\lceil \varepsilon \delta_n \rceil$ vertices $S \subseteq N_1(v)$  and label them $v_1, \ldots, v_{\lceil \varepsilon \delta_n \rceil}$. Notice that the $v_i$ have at least $\delta_n -1$ unique neighbors in $N_2(v)$. Ignoring the edges in $N_2(v)$ gives us the bipartite graph $J$ with parts $S$ and $N_2(v)$ in Lemma~\ref{lem:H} with parameters $d = \delta_n-1$ and $a = \lceil \varepsilon \delta_n \rceil$. Hence, define $\lbrace J_{k,n} \rbrace_{k=1}^r$ as the $r$ spanning subgraphs of $G_n$ guaranteed by Lemma~\ref{lem:H}. In addition, we have that $r$ is at least 
$$\binom{(1+\varepsilon)\delta_n-2}{\delta_n-2} \sim 2^{((1+\varepsilon)\delta_n-2)H(1/(1+\varepsilon))}$$ 
since
$$
\frac{d-1}{a+d-1} \sim \frac{1}{1+\varepsilon}.
$$
Furthermore, since each subgraph is connected and has at most $\lceil \varepsilon\delta_n \rceil$ many vertices with degree $o(\sqrt{n})$ and $\delta_n - \varepsilon \delta_n = \Theta(\sqrt n)$, Lemma~\ref{lem:subgraph} implies that $\lbrace J_{k,n}\rbrace_{n\in I}$ is a Meyniel extremal family for each $1\le k \le r$. 
\end{proof}

\begin{theorem}\label{thm:spanning4}
Fix $0 <\varepsilon < 1$. If $G_n$ is a $C_4$-free graph with degrees in $\Theta(\sqrt{n})$ and $\delta_n =\delta(G_n)$, then there exists $r$ pairwise nonisorphic, Meyniel extremal, spanning families of $\lbrace G_n \rbrace_{n \in I}$, where 
$$r \ge \binom{(1+\varepsilon)\delta_n-3}{\delta_n-3} \sim 2^{((1+\varepsilon)\delta_n-3)H(1/(1+\varepsilon))}$$ 
and $H(x)$ is the binary entropy function.
\end{theorem}
\begin{proof}
Let $v$ be a vertex of maximum degree in $G_n$ and label $N_1(v) = \lbrace v_1, \ldots, v_{\lceil \varepsilon\delta_n \rceil} \rbrace$. As in the proof of Corollary~\ref{cor:vec}, we may make $N_1(v)$ an independent set, choose the $\lceil \varepsilon \delta_n \rceil$ from $N_1(v)$ and label them $v_i$. Each $v_i$ have at least $\delta_n -2$ unique neighbors in $N_2(v)$. We then use Lemma~\ref{lem:H} with parameters $d = \delta_n-2$ and $a = \lceil \varepsilon \delta_n \rceil.$ Following a similar application of Lemma~\ref{lem:subgraph} as in the proof of Theorem~\ref{thm:spanning5}, we have that $G_n$ contains spanning subgraphs $\lbrace J_{k,n}\rbrace_{k=1}^r$ where 
$$ r \ge \binom{(1+\varepsilon)\delta_n-3}{\delta_n-3} \sim 2^{((1+\varepsilon)\delta_n-3)H(1/(1+\varepsilon))}$$ 
and each family $\lbrace J_{k,n}\rbrace_{n\in I}$ is a Meyniel extremal family. As in Theorem~\ref{thm:spanning5}, we have that
$$
\frac{d-1}{a+d-1} \sim \frac{1}{1+\varepsilon}
$$
The proof follows.
\end{proof}

In both of these proofs, we use Lemma~\ref{lem:H} to focus solely on edges between the first and second neighborhoods of a vertex of minimum degree. Although this is a local approach, it allows us to conclude that if two starting families are nonisomorphic due to some global structure, such as the presence of triangles, then all of the $2r$ families guaranteed by Theorem~\ref{thm:spanning5} and \ref{thm:spanning4} will be pairwise nonisomorphic. More generally, if two families are nonisomorphic due to some structure that avoids the first and second neighborhood of a minimum degree vertex (suitably chosen for each corresponding member), then the $2r$ resulting families will be pairwise nonisomorphic. 

For example, take some $C_4$-free, bipartite, Meyniel extremal family $\lbrace G_n \rbrace_{n \in I}$. Each of the families $\lbrace G_{i,n} \rbrace_{n\in I}$ given by Lemma~\ref{lem:fact} are nonisomorphic due to their graphs being $d$-regular for different $d$. In addition, each family satisfies the conditions of Theorem~\ref{thm:spanning4}. Thus, all of the families generated by taking $\lbrace G_{i,n} \rbrace_{n\in I}$ as a starting family will be pairwise nonisomophic. 

These observations apply to the following two corollaries, whose starting families are nonisomorphic over the presence of triangles. 

\begin{corollary}
Fix $0 <\varepsilon < 1$ and let $I$ be the set of prime powers. If $G_q$ is the incidence graph of the projective plane on $2(q^2 + q + 1)$ vertices for a prime power $q$, then there exists $r$ pairwise nonisomorphic, Meyniel extremal families of spanning subgraphs of $\lbrace G_q \rbrace_{q \in I}$, where $ r \ge(1+o(1)) 2^{((1+\varepsilon)(q+1)-2)H(1/(1+\varepsilon))}$
for $H(x)$ the binary entropy function.
\end{corollary}
\begin{proof}
The graph $G_q$ is regular with vertices of degree $q+1$. Therefore, Theorem~\ref{thm:spanning5} implies the result.
\end{proof}

\begin{corollary}
Fix $0 < \varepsilon < 1$ and let $I$ be the set of prime powers. Let $G_q$ be a graph that has order $q^2 + q + 1$, each vertex has degree $q+1$ or $q$, is $C_4$-free, and has diameter~$2$. There then exists $r$ pairwise nonisomorphic, Meyniel extremal, spanning families of $ \lbrace G_q \rbrace_{q\in I}$, where $ r \ge (1+o(1))2^{((1+\varepsilon)q-3)H(1/(1+\varepsilon))}$
for $H(x)$ the binary entropy function.
\end{corollary}
\begin{proof}
By Theorem~3.1 in~\cite{bb}, the family $\lbrace G_q \rbrace_{q}$ is Meyniel extremal. Theorem~\ref{thm:spanning4} yields the result.
\end{proof}

In light of the observations above, the following corollary allows us to take a starting family and generate nonisomorphic, families of spanning subgraphs that have a specified number of triangles. This is done such that families with different number of triangles will produce pairwise nonisomorphic families. 

\begin{corollary}
Fix $0 <\varepsilon < 1$. Let $G_n$ be a $C_4$-free graph with $t_n$ triangles and degrees in $\Theta(\sqrt n)$. Let $\delta_n =\delta(G_n)$. There then exists $r$ pairwise nonisomorphic, Meyniel extremal, spanning families of $\lbrace G_n \rbrace_{n \in I}$, where 
$$r \ge \binom{(1+\varepsilon)\delta_n-3}{\delta_n-3} \sim 2^{ ((1+\varepsilon)\delta_n-3)H(1/(1+\varepsilon))},$$ 
the function $H(x)$ is the binary entropy function, and each member of the new families has exactly $t_n'$ triangles, with $0 \le t_n' \le t_n$.
\end{corollary}
\begin{proof}
Denote the set of triangles in $G_n$ as  $\lbrace T_i\rbrace_{i=1}^{t_n}$ where $T_i$ is on $x < y < z$. For the vector $\mathbf x\in\{0,1,2\}^{t_n'}$, let $G_{\mathbf {x},n}$ be obtained by removing one edge from $T_i$ for $1 \le i \le t_n'$ in the following way. If $\mathbf x = (a_i)_{i=1}^{t_n'}$ and $a_i=0$, then remove the edge $xy$ from $T_i$. If $a_i =1$, then remove the edge $yz$; if $a_i=2$, then remove $xz$. Notice that each vertex has degree at least $\delta_n/2 = \Theta(\sqrt n)$. For this, consider a vertex $v$. Each edge in $G_n$ not contained in a triangle will remain in $G_{\mathbf{x},n}$, so $v$ loses at most one edge for each triangle containing $v$. As all the triangles in $G_n$ are edge-disjoint, there are at most $\deg_{G_n}(v)/2$ triangles containing $v$. 

Now take a vertex of minimum degree $v$ in $G_n$.
Apply the above deletion procedure to each $G_n$ until we have $t_n'$ triangles remaining. Observe that all edges in $G_n$ within $N_1(v)$ are contained in triangles in $G_n$ containing $v$, so we may choose $\mathbf x$ such that only the edges in $N_1(v)$ are removed. In addition, we may choose $\mathbf x$ in such a way that all edges from $N_1(v)$ to $N_2(v)$ remain. Define a family of bipartite graphs $\lbrace J_n \rbrace_{n\in I}$ such that one bipartition consists of $\lceil \varepsilon \delta_n\rceil$ vertices of $N_1(v)$ in $G_n$ and the other is $N_2(v)$. 
 
We may apply Lemma~\ref{lem:H} to $J_n$ (ignoring the edges in $N_2(v)$) to obtain nonisomorphic families $\lbrace J_{k, n}\rbrace_{n\in I}$ for $1 \le k \le r$ with 
$$r \ge \binom{(1+\varepsilon)\delta_n-3}{\delta_n-3} \sim 2^{ ((1+\varepsilon)\delta_n-3)H({1}/{(1+\varepsilon))}},$$
as
$$
\frac{d-1}{a+d-1} \sim \frac{1}{1+\varepsilon}.
$$
Note that $J_{k,n}$ is a spanning subgraph of $J_n$ (and thus, of $G_n$). 
In addition, since each subgraph is connected and has at most $\varepsilon \delta_n$ many vertices with degree $o(\sqrt{n})$, Lemma~\ref{lem:subgraph} implies that each of these spanning subgraphs are Meyniel extremal. Since $J_{k,n}$ also contains exactly $t_n'$ triangles, the result follows.  
\end{proof}

\section{Techniques from hypergraphs}

A \emph{hypergraph} is a discrete structure with vertices and \emph{hyperedges}, which consists
of sets of vertices. Graphs are special cases of hypergraphs, where each hyperedge has cardinality two. A \emph{blocking set} of a hypergraph $(V,E)$ is a subset of its vertices such that each edge contains one vertex from the subset of vertices. Define the indicator variable $x_v$ to be $1$ if $v$ is in the blocking set. The condition 
\[
\sum_{v \in e} x_v \geq 1,
\]
holds for each edge $e$ in the hypergraph. We can then think of finding a minimum cardinality of a blocking set as an IP problem, with an objective function
\[
\sum_{v \in V} x_v,
\]
which is being minimized. The minimum value of this objective function will be denoted $\tau$. 
We may relax $x_v$ to be a nonnegative real value, in which case the IP problem becomes an LP one. The resulting minimum value of the objective function is written $\tau^*$, and the solution is known as a \emph{fractional solution to the blocking problem}. 

\begin{theorem}[Lov\'{a}sz \cite{lovasz1975}] \label{thm:lovasz}
For a hypergraph $(V,E)$, let $\tau$ denote the cardinality of a minimum cardinality blocking set, $\tau^*$ denote the minimum value of a fractional solution of the blocking set, and $d$ the maximum degree of a vertex. We then have that
\[
\tau < \tau^* (1 + \log d).
\]
\end{theorem}

\subsection{Domination number}

Theorem \ref{thm:lovasz} can be used to prove results on graph domination, and the aim of this subsection is to use it to prove the following. 

\begin{theorem} \label{thm:dominationMostLargeDeg}
Let $\omega=\omega(n)$ be a nondecreasing, integer-valued function tending to infinity, and let $\lbrace G_n \rbrace_{n\in I}$ be a family of graphs, where $G_n$ is of order $n$. If $G_n$ has at most $O(\sqrt{n})$ vertices of degree $o(\sqrt{n}),$ then $G_n$ has domination number $O(\omega \sqrt{n} \log n )$.
\end{theorem}

We have the following corollary.

\begin{corollary} \label{cor:resultsMey}
If Meyniel's conjecture is false, then either any family of graphs has cop number $O(\sqrt{n} \log n)$, or there is a family of graphs $\lbrace G_n \rbrace_{n\in I}$, where $G_n$ has order $n$, such that $G_n$ has cop number $\omega(\sqrt{n}\log n)$, and $\omega(\sqrt{n})$ vertices of degree $o(\sqrt{n})$.
\end{corollary}

By Corollary~\ref{cor:resultsMey}, we know that a family of graphs that violates Meyniel's conjecture may have a nontrivial number of vertices with small degree. 
In the search for graphs with the asymptotically largest cop number (in particular, if the aim is to find graphs with cop number $\omega(\sqrt{n})$), it may therefore be important to consider those graphs with some vertices of a smaller degree. This is at odds with the current constructions of graph families with high cop number, as these typically have all vertices of approximately the same degree, and each vertex has a relatively large degree. 

Another consequence of Theorem~\ref{thm:dominationMostLargeDeg} is that if the soft Meyniel's  conjecture is true, and there exists a graph family with cop number $\Theta(n^{1-\alpha})$ for some $\alpha >0$, then we must have the minimum degree of a graph of the graph class to be at most $n^{\alpha} (1+\log\Delta)$. This may be asymptotically much smaller than $\Theta(n^{1/2})$, although so far we only require one vertex to have such a small degree. Theorem~\ref{thm:dominationMostLargeDeg} therefore strengthens this result, showing that such a family of graphs requires a significant number of such vertices. 

\medskip

The idea underlying our proof of Theorem~\ref{thm:dominationMostLargeDeg} is to construct a certain hypergraph, where each hyperedge is formed as the closed neighborhood of a given vertex of the graph. The set of vertices that form a blocking set of the hypergraph is also a dominating set of the original graph. We note that this is a modification of an idea by Beke \cite{Beke2022} to determine an upper bound on the metric dimension of incidence graphs of M\"{o}bius planes. 

To demonstrate the hypergraph technique, we prove the following result as a warm-up, although stronger results are known to exist; see Theorem~1.2.2 of \cite{as}, for example.  
\begin{theorem} \label{thm:upperPrecise}
If $G$ is a graph of order $n$, then the domination number of $G$ is at most 
\[\frac{n}{\delta(G)} (1+\log\Delta(G)).\] 
\end{theorem}
\begin{proof}
Define a hypergraph with vertex set $V$ and with edges defined for each $w \in V$ as $h_w = \{ v \in V : v \in N(w)\}$. 
The hypergraph has $d = k_1$. 
Note that if we find a blocking set $S$ of this hypergraph, then $S$ is also a subset of $V$ such that every vertex in $V$ is adjacent in $G$ to at least one vertex in $S$. In particular, $S$ dominates $V$ (although note that every vertex of $V$ is adjacent to a vertex in $S$, not just $V\setminus S$).  
In the related LP problem, set $x_v = \frac{1}{\delta(G)}$. We then have that $\sum_{v \in h_w} x_v \geq 1$ for each hyperedge $h_w$, and the objective function evaluates to $\tau^* = \frac{|V|}{\delta(G)}$. Thus, we have that
$$\tau <  \frac{|V|}{\delta(G)} (1+\log\Delta(G)),$$ and so 
any minimal set that dominates $V$ has cardinality at most $\frac{|V|}{\delta(G)} (1+\log\Delta(G))$. The proof follows.
\end{proof}

We now turn to the proof of the main result of the section.

\begin{proof}[Proof of Theorem~\ref{thm:dominationMostLargeDeg}]
For convenience, assume $k$ defined as $(\omega(n))^k = n^{1/2}$ is an integer. (The proof is straightforwardly modified,  otherwise.)  
Consider $G$, a graph in the family of graphs, that has $n$ vertices. 
Suppose $V_i$ is the subset of vertices $u$ such that  $(\omega(n))^{i} < \mathrm{deg}(u) \leq (\omega(n))^{i+1}$ for $1 \leq i \leq k-1$, $V_0$ is the subset of vertices $u$ such that $1 \leq \mathrm{deg}(u) \leq \omega(n)$,  and $V_k$ is the subset of vertices $u$ such that $n^{1/2} < \mathrm{deg}(u)$. 
Note that by assumption, $\sum_{i=0}^{k-2} |V_i| = O(\sqrt{n})$. 

Define a hypergraph with vertices $V(G)$ and with hyperedge $h_w=N_G(w)$ for each $w \in V(G)$. 
Define the function $s(u) = \min_{v \in N(u)} \text{deg}(v)$.
 It follows that the vertices $u$ that satisfy $$(\omega(n))^{-(i+1)} \leq \frac{1}{s(u)} < (\omega(n))^{-i},$$ for $0 \leq i \leq k-2$, are contained in a subset of $N(V_i)$, and so there are at most $|V_i| (\omega(n))^{i+1}$ such vertices. 
 There are also at most $n$ vertices $u$ that satisfy $$(\omega(n))^{-k} \leq \frac{1}{s(u)} < (\omega(n))^{-k+1} = n^{-1/2}\omega(n),$$ and at most $n$ vertices $u$ that have $\frac{1}{s(u)} < n^{-1/2}$, which are subsets of $N(V_{k-1})$ and $N(V_{k})$, respectively. 
 
We now define the variables in the corresponding LP problem.  For each vertex $v$ in the hypergraph, we may define $x_v = \frac{1}{s(u)}$. It follows that $\sum_{v \in h_w}  x_v \geq 1$ for each $w \in V(G)$. 
We then have that
 \begin{align*}
 \tau^* 
 &\leq \sum_{i=0}^{k} \sum_{v \in V_i} \frac{1}{s(u)} \\
 &<  |V_{k-1}| n^{-1/2}\omega(n)+ |V_k| n^{-1/2} +  \sum_{i=0}^{k-2} |V_i| (\omega(n))^{i+1} (\omega(n))^{-i} \\
 &=  
 n^{1/2} +  n^{1/2}\omega(n) + \sum_{i=0}^{k-1} O(n^{1/2}) \omega(n)\\
 & = O(n^{1/2}\omega(n)).
 \end{align*}
 The proof now follows by Theorem~\ref{thm:lovasz}. 
\end{proof}

\subsection{Cop number}

Frankl~\cite{frankl} proved that for a graph $G$, $c(G) \leq (1 + o(1)) \frac{n \log \log n}{\log n}$, which was improved to $O(\frac{n}{\log{n}})$ in \cite{ch}. The results of \cite{fkl,lu,ss} reduced this upper bound down further to 
\begin{equation}
c(n) \leq O\left( \frac{n}{2^{(1-o(1)) \sqrt{\log_2{n}}}}\right). \label{bkb}
\end{equation}
We give new upper bounds on the cop number, improving (\ref{bkb}) in some cases, using the hypergraph techniques in this section.

As first defined in \cite{ch}, a \emph{minimum distance caterpillar} (or \emph{mdc})  in a graph $G$ is an induced subgraph of $G$ whose vertices consist of a shortest path $P$ between two vertices of $G$, along with a subset of the neighbors of vertices in $P$. We say the \emph{length} of an mdc is the length of $P$. As shown in \cite{ch}, five cops may \emph{guard} an mdc, in the sense that after some number of rounds, the robber is captured if they enter it. Define a \emph{diameter length caterpillar} (or \emph{DLC}) as an mdc of length $\mathrm{diam}(G)$. In particular, five cops may guard a DLC.

A graph is \emph{vertex-transitive} if for every two vertices $x$ and $y$, there is an automorphism mapping $x$ to $y.$ A vertex-transitive graph is $m$-regular for some nonnegative integer $m,$ and we refer to $m$ as its \emph{degree}. The following theorem bounds the cop number of a vertex-transitive graph by a function of its degree and diameter. 

\begin{theorem} \label{thm:vertexTrans}
Let $G$ be a vertex-transitive graph $G$ with degree $m$, and let $d = m \cdot \mathrm{diam}(G)$. We then have that 
\[
c(G) \leq  \frac{3n \log{d}}{d} = O\left(\frac{n \log{d}}{d}\right).
\]
\end{theorem}
\begin{proof}
Let $V$ be the vertices in $G$. 
Let $\mu_1$ be the number of DLCs that any one vertex is in (note that by the vertex-transitivity of $G$, this value is the same for any choice of vertex). Let $\mu_2$ be the minimum number of vertices in any DLC. 
The maximum cardinality of a DLC is at most $d$, so $\mu_2 \leq d$. 

Define a new set of vertices $W,$ where each $w\in W$ is associated with a DLC, $m_w$, of $G$. 
Define a hypergraph $\mathcal{H} = (W,H)$ on the vertex set $W$ by including the hyperedge $h_v = \{w \in W : v \in m_w\}$ in the hyperedge set $H$, for each $v \in V$. In particular, a hyperedge $h_v \in H$ contains vertex $w\in W$ exactly when the vertex $v \in V$ is contain in the DLC $m_w$ of $G$. 
Let $x_w = \frac{1}{\mu_1}$. We then have that $\sum_{w \in h} x_w = |h|/{\mu_1} = 1$ for each hyperedge $h$, so the conditions of the LP are satisfied. 

We now find an upper bound for $|W|$. 
Consider the double count of the pair $\{(w,v)  \in W \times V : w \in h_v\}$, which gives
\[
\sum_{v \in V} |h_v| = \sum_{w \in W} |m_w|. 
\]
The left-hand side is $\mu_1 n$. 
As $|m_w|\geq \mu_2$, the right side is greater than $|W| \mu_2$. 
We then have that $|W| \leq \frac{\mu_1 n}{\mu_2}$. 

As a result, 
\begin{align*}
    \tau^* &= \sum_{w \in W} \frac{1}{\mu_1}
    \leq |W| \frac{1}{\mu_1}
    = \frac{n}{\mu_2}. 
\end{align*}
By Theorem~\ref{thm:lovasz}, this gives $\tau <\frac{n\log{d}}{\mu_2}$. 

The set of vertices in a blocking set of $\tau$ vertices corresponds to a set of DLCs in $G$ such that every vertex of $G$ is covered by some DLC. Since five cops may guard each of these DLCs, $5 \tau$ cops is sufficient to capture the robber in $G$.

Let $P$ be the path of length $\mathrm{diam}(G)$ associated with some DLC, and let $P'$ be a subset of the vertices of $P$ such that any pair of vertices in $P'$ have distance at lease $2$ from each other. Note that $P'$ contains $\lceil\mathrm{diam}(G)/3\rceil \geq \mathrm{diam}(G)/3$ vertices. 
Now $u,v\in P'$ are also vertices in the DLC, and have distance at least two in $G$. As such, $N(u) \cap N(v) = \emptyset$. as such, the set of all neighbors of vertices in $P'$ is equal to $m \cdot P'\geq m\cdot \mathrm{diam}(G)/3$. 
This implies that the number $\mu_2$ must be larger than $m\cdot \mathrm{diam}(G)/3 = d/3$, and so we have that $\tau < \frac{n \log d}{d/3}$, and the result follows. 
\end{proof}

The upper bound in Theorem \ref{thm:vertexTrans} provides the best-known bounds on the cop number for vertex-transitive graphs when the degree is not too small, improving the one in (\ref{bkb}). 

\begin{corollary}\label{bbb}
Suppose that $\omega=\omega(n)$ is a nondecreasing, integer-valued function function tending to infinity, and $G$ is a vertex-transitive graph with degree $m$ and with $$m \cdot \mathrm{diam}(G)\ge 3\omega
2^{(1-o(1))\sqrt{\log_2{n}}}\sqrt{\log_2{n}}.$$ The following inequalities then hold for $n$ sufficiently large.
\begin{enumerate}
    \item If $\omega = \Omega\left(2^{\sqrt{\log_2{n}}}\right)$, then $$c(G) \leq  \frac{n}{2^{(1-o(1))\sqrt{\log_2{n}}}} \frac{3\log{\omega}}{3\omega \sqrt{\log_2{n}}}=o\left(\frac{n}{2^{(1-o(1))\sqrt{\log_2{n}}}}\right).$$
    \item If $\omega = o\left(2^{\sqrt{\log_2{n}}}\right)$, then 
    \[
c(G) \leq \frac{n}{ 2^{(1-o(1))\sqrt{\log_2{n}}}} \frac{1-o(1)}{3\omega} =o\left(\frac{n}{2^{(1-o(1))\sqrt{\log_2{n}}}}\right).
\]
\end{enumerate}
\end{corollary}

We also have that Meyniel's soft conjecture holds up to a logarithmic factor for vertex-transitive graphs of order $n$ when the degree is a fractional power or linear in $n$. 
\begin{corollary}
If $G$ is a vertex-transitive graph with degree $m=\Theta(n^{1-\varepsilon})$ for a constant $\varepsilon\geq 0$, then 
\[
c(G) = O(n^{1-\varepsilon} \log n).
\]
\end{corollary}

The proof technique of Theorem \ref{thm:vertexTrans} can be generalized for a larger class of graphs that are less regular than vertex-transitive graphs. 
In a graph $G,$ define $\mu_1$ to be the maximum number of DLCs that contain any one vertex. 
We define two parameters $\sigma$ and $e_1$ such that most vertices are contained in between $\mu_1 /\sigma$ and $\mu_1$ DLCs, with exactly $e_1$ exceptional vertices that are contained in less than $\mu_1 /\sigma$ DLCs. These exceptional vertices can be covered by placing a cop on each of them, and so can be ignored for the proof. By allowing vertices to be contained in different numbers of DLCs, the upper bound changes to $c(G)  \leq O \big(\frac{n \sigma \log{d}}{\mu_2}\big)$. 
We also define $e_2$ and $\mu_2$ such that all DLCs contain $\mu_2$ vertices not in the $e_1$ exceptions mentioned previously, except for exactly $e_2$ exceptional DLCs that have less than $\mu_2$ such vertices. This modifies Theorem \ref{thm:vertexTrans} by allowing $\mu_2$ to be potentially much higher, as the small number of exceptions can each be covered by a set of five cops. 
As long as a number of conditions still hold for $e_1$ and $e_2$, we will have that $c(G)  \leq O \big(\frac{n \sigma \log{d}}{\mu_2}\big)$. 

\section{New families of Meyniel extremal graphs with prescribed properties} 

This section provides new constructions of Meyniel extremal families whose graphs are regular with large chromatic number, and ones with diameter at least some fixed constant. We give a construction for creating bipartite Meyniel extremal families from any given Meyniel extremal family.

\subsection{Regular with large chromatic number}
We consider a method using graph products to construct new Meyniel extremal families consisting of regular graphs with various properties such as high chromatic number or clique number. For more on graph products, the reader is directed to \cite{ik}. For graphs $G$ and $H,$ define their \emph{lexicographic product}, written $G\bullet H,$ to have vertices $V(G) \times V(H),$ and $(u,v)$ is adjacent to $(x,y)$ if $u$ is adjacent to $x$ in $G$, or $u=x,$ and $v$ is adjacent to $y$ in $H.$ We may think of $G\bullet H$ as replacing each vertex $x$ of $G$ with a copy of $H$ labeled as $H_x,$ such that if $xy\in E(G)$, then all edges are present between $H_x$ and $H_y.$ Note that the order of $G\bullet H,$ is $|V(G)||V(H)|.$ Schr\"oder \cite{schr} proved that if $c(G) \ge 2,$ then $c(G\bullet H) = c(G).$

\begin{theorem}\label{tlp}
Suppose that $\lbrace G_n \rbrace_{n\in I}$ is a Meyniel extremal family and $H$ is a fixed graph. We then have that $\lbrace G_n \bullet H \rbrace_{n\in I}$ is a Meyniel extremal family.
\end{theorem}

\begin{proof} Suppose that $|V(G_n)| = n,$ $|V(H)|=m,$ and let $D>0$ be a constant such that $c(G_n) \ge D\sqrt{n}$ for all $n$. We then have that $c(G_n \bullet H) \ge D\sqrt{n}$ and the order of $G_n\bullet H$ is $nm.$

Hence, $c(G_n \bullet H) \ge D'\sqrt{|V(G_n\bullet H)|},$ with $D' =D/\sqrt{m}.$ \end{proof}

\begin{corollary}
For an integer $t \ge 1,$ there exist Meyniel extremal families containing graphs that are regular and with clique and chromatic number at least $t.$
\end{corollary}
\begin{proof} Apply Theorem~\ref{tlp} with $H=K_t$, the complete graph of order $t$, and $\lbrace G_n \rbrace_{n\in I}$ the family of incidence graphs of projective planes. As $H$ is an induced subgraph of $G_n \bullet H$, the result follows.
\end{proof}

With our approach, we may also find Meyniel extremal graphs with bounded clique number and chromatic number at least a fixed constant: choose $H$ to have sufficiently large girth and chromatic number. By taking $H$ to be a graph with no edges, we may also find Meyniel extremal families of regular graphs with independence number larger than any fixed constant.

\subsection{Large diameter}

In most cases of Meyniel extremal families, the graphs are of small diameter. For example, the incidence graphs of projective planes have diameter 3, while polarity graphs have diameter $2$. In this subsection, we find Meyniel extremal families of graphs with any constant even diameter by using a variation of incidence graphs of projective planes.   

We specify two parameters, $q$ a prime power and $m$ a positive integer. 
Let $(X,\mathcal{B})$ be a projective plane of order $q$, which is known to exist for such $q$. For each block $B \in \mathcal{B}$ of cardinality $q+1$, we define $B',B''$, each of cardinality $\frac{q+1}{2}$, such that $B = B'\cup B''$. 
Write the corresponding blocksets as $\mathcal{B}'$ and $\mathcal{B}''$. 

We define a tripartite graph on vertex set $X\cup \mathcal{B}' \cup \mathcal{B}''$, and where an edge connects $u\in X$ to $v \in \mathcal{B}'\cup \mathcal{B}''$ if $u \in v$. 
Orient a cycle $C_m$ such that each vertex has in-degree and out-degree $1$. 
Construct a \emph{blow up} of $C_m$, say $H$, where each vertex $v$ of $C_m$ is replaced with $q^2+q+1$ vertices $\{(v,B) : B \in \mathcal{B}\}$ and each edge $e$ of $C_m$ is replaced with $q^2+q+1$ vertices $\{(e,x) : x \in X\}$. 
Vertices $(u,B)$ and $(e,x)$ of $H$ are adjacent if $e$ is an out-edge of $u$, $B\in \mathcal{B}'$ and $x \in B$; or if $e$ is an in-edge of $u$, $B\in \mathcal{B}''$ and $x \in B$. 
The resulting graph is denoted by $\mathrm{BF}(q,m)$. 

\begin{theorem}\label{ttt}
Let $q$ be a prime power and $m$ a positive integer. The graph $\mathrm{BF}(q,m)$ is a bipartite graph with the following properties:
\begin{enumerate}
\item order $2(q^2+q+1)m$ and $(q^2+q+1)(q+1)m$ edges;
\item $C_4$-free; 
\item diameter $2m$;
\item $(q+1)$-regular; and 
\item cop number at least $q+1$. 
\end{enumerate}
\end{theorem}
\begin{proof}
For each vertex and each edge in $C_m$, $q^2+q+1$ vertices are created, giving $(q^2+q+1)(|V(G)|+|E(G)|)$ vertices in the blow up. 
A copy of the tripartite graph replaces each edge in $G$, and each tripartite graph has $2(q^2+q+1)$ block vertices, each of degree $(q+1)/2$, giving $(q^2+q+1)(q+1)m$ edges in the constructed graph. Hence, item (1) holds.

The tripartite graph that replaces each edge is $C_4$-free, since the incidence graph of the projective plane was $C_4$-free. 
If a $C_4$ existed between two tripartite graphs that replaced two adjacent edges, then the incidence graph of the projective plane would contain a $C_4$, giving a contradiction. Item (2) holds.
 
Each vertex that was the result of blowing up a vertex of $C_m$ has degree $2 \left(\frac{q+1}{2}\right)=q+1$, and each vertex that was the result of blowing up an edge of $C_m$ has degree $q+1$. The diameter of the graph is twice that of $C_m$. Items (3) and (4) hold.
Item (5) follows by Corollary~\ref{cor11}, and the theorem follows.
\end{proof}

The following corollary follows from Theorem~\ref{ttt}.

\begin{corollary}
For every positive integer $d\ge 3$ there exists a Meyniel extremal family whose graphs are regular and have diameter at least $d$. 
\end{corollary}

\subsection{Bipartite Meyniel extremal graphs}

Let $G$ be graph that is not bipartite. The \emph{categorical product} $G \times K_2$ has vertex set $\{(v,a) : v\in V(G), a\in V(K_2)\}$ and an edge between $(v,a)$ and $(v',a')$ when $vv' \in E(G)$ and $aa' \in E(K_2)$.  
The graph $G\times K_2$ is called the \emph{bipartite double cover} or \emph{Kronecker cover} of $G$, and is denoted by $B(G).$ Note that $B(G)$ has twice the number of vertices as $G$. 
The degree of vertex $(v,a)$ in $B(G)$ is the same as the degree of $v$ in $G$. Further, $B(G)$ is connected as $G$ contains at least one cycle of odd length (since it is not bipartite), and is bipartite. 

The following theorem provides bounds on the cop number of a bipartite double cover.

\begin{theorem} \label{thm:BlowUp}
For a graph $G$, we have that \[c(G) \leq c(B(G)) \leq 2c(G).\] 
\end{theorem}
\begin{proof}
We begin by showing the lower bound.  When playing Cops and Robbers on $B(G)$, the robber can apply the mapping $f:B(G) \rightarrow V(G)$ defined as $f(v,a)=v$ to the location of each cop to convert the game to a game on $G$. 
The robber can then use the strategy defined on $G$ to avoid the robber. Label the vertices of $K_2$ by $0$ and $1.$
If this strategy tells the robber to move from vertex $u$ to $v$ in $G$, then the robber moves from $(u,a)$ to $(v,a+1) \pmod{2}$ in $B(G)$. 
The robber in $B(G)$ is captured only if the robber in $G$ is captured. 
Therefore, we have that $c(B(G)) \geq c(G)$. 

Now we show the upper bound. Suppose that cops $C_1, C_2,\ldots, C_k$ capture the robber on $G$. 
To each cop on $G$, we associate two cops playing on $c(B(G))$. 
During each round and for $1\le i \le k$, the two cops $C_{i,0},C_{i,1}$ associated to $C_i$ play on the vertices $(C_i,0)$ and $(C_i,1)$. 
The robber is captured by the cops $C_{i,0},C_{i,1}$ on $c(B(G))$ when the robber is captured on $G$ by the cops $C_i$. 
\end{proof}

We have the following immediate corollary, which gives a strong case to focus on the cop number of bipartite graphs. 

\begin{corollary}\label{cor1}
If there exists a family of connected graphs $\lbrace G_n \rbrace_{n\in I}$ with $c(G_n) = \Theta(n^{1-\epsilon})$, where $0 \leq \epsilon <1 $, and $|V(G_n)| = \Theta(n)$, then there exists a family of bipartite connected graphs $\lbrace G_n' \rbrace_{n\in I}$ with $c(G_n') = \Theta(n^{1-\epsilon})$ and $|V(G_n')| = \Theta(n)$.
\end{corollary}

In particular, Corollary~\ref{cor1} shows that every Meyniel extremal family gives rise to a Meyniel extremal family whose members are bipartite. Further, if there is a family of connected graphs that violate Meyniel's conjecture, then we can find a family of connected bipartite graphs that violate Meyniel's conjecture.

We note that the number of $C_4$'s in $B(G)$ will be twice the number in $G$. The number of $C_6$'s in $B(G)$ will be the number of triangles in $G$ plus twice the number of $C_6$'s in $G$. Hence, if $G$ is $C_4$-free, then so is $B(G)$, and if $G$ is triangle-free and $C_6$-free, then $B(G)$ is $C_6$-free. 

Our approach using the bipartite double cover also yields another short proof of the result first proven in \cite{bb}.
\begin{corollary}
If $G$ is a $C_4$-free graph, then $c(G) \geq \delta(G)/2.$
\end{corollary}
\begin{proof}
If $G$ is $C_4$-free, then $B(G)$ has girth at least $6$, and so has cop number $c(B(G))\geq \delta(B(G)) = \delta(G)$. 
By Theorem~\ref{thm:BlowUp}, this gives that $\delta(G) \leq 2c(G)$, and the result follows. 
\end{proof}

\section{Further directions}

All known Meyniel extremal families, such as the incidence graphs of projective planes and polarity graphs, rely on properties of degrees and on forbidding subgraphs. We demonstrated that the minimum degree of graphs in Meyniel extremal families may vary widely: Corollary~\ref{cor:vec} showed that there exist classes of Meyniel extremal families containing graphs on $n$ vertices with $\Theta(\sqrt{n})$ vertices of constant degree. In our constructions of Meyniel extremal families in Sections 2 and 5, both the maximum degree and the average degree remain high. 
In the recent paper~\cite{hos}, it was shown that for every $\varepsilon > 0$, there exists a graph with maximum degree $3$ on $n$ vertices with cop number $\Omega(n^{1/2-\varepsilon})$. Their methods use a degree-reducing technique that seems to break down when trying to obtain a Meyniel extremal family. 

Our discussion suggests the following.
\vspace{.1in}

\noindent \textbf{Maximum degree conjecture}: Every Meyniel extremal family contains graphs with maximum degree $\omega(1)$.
\vspace{.1in}

If the Maximum Degree Conjecture is true, then there is no Meyniel extremal family of subcubic graphs (that is, graphs whose degrees are at most 3). Corollary~\ref{cor:vec} as well as Lemma~\ref{lem:newlb-girth} and ~\ref{lem:newlb-K2t} suggests that the number of vertices of certain degrees must be controlled to maintain Meyniel extremality. We also conjecture that the average degree must be unbounded, which would imply the previous conjecture.
\vspace{.1in}

\noindent \textbf{Average degree conjecture}:
Every Meyniel extremal family contains graphs with average degree $\omega(1)$.
\vspace{.1in}

Many of the known examples do not deviate from an average degree of $\Theta(\sqrt n)$, and our current efforts seem to suggest that this is an important threshold. The analysis of the cop number of the binomial random graph $G(n,p)$ in \cite{lp} suggests that there are examples with lower average degree: their results may be extended to when the average degree is $n^{1/(2k) + o(1)}$ for natural numbers $k$ and achieve a cop number of $\Theta(\sqrt n)$, with the possible additional factor of $\log^{O(1)} n$. An open problem is to find a Meyniel extremal family containing graphs with average degree $o(\sqrt n)$.

Forbidding too many short cycles is not possible if the degree is too high. A Moore bound argument gives that if a graph $G$ has $n$ vertices and $\delta(G) = \Theta(\sqrt{n})$, then $G$ has girth at most $6$. We must therefore search for graphs with smaller minimum degree to find Meyniel extremal families whose graphs do not contain any $C_6$. Finding a Meyniel extremal family whose members are $C_6$-free remains an open problem. Interestingly, one of the best-known examples of a Meyniel extremal family, the incidence graphs of projective planes, has the largest possible number of $C_6$'s in a $C_4$-free, balanced bipartite graph; see~\cite{fio}.

\section{Statements and Declarations}
The authors were supported by NSERC. The authors have no relevant financial or non-financial interests to disclose. All authors contributed equally to the paper.

\end{document}